\documentclass[a4paper,11pt]{amsart}

\usepackage[utf8]{inputenc}
\usepackage[T1]{fontenc}
\usepackage{lmodern}
\usepackage{microtype}
\usepackage{hyperref}
\usepackage{geometry}

\usepackage[abbrev]{amsrefs}

\newtheorem{proposition}{Proposition}[section]
\newtheorem{theorem}[proposition]{Theorem}
\newtheorem{lemma}[proposition]{Lemma}

\theoremstyle{definition}

\theoremstyle{remark}
\newtheorem{remark}{Remark}[section]

\numberwithin{equation}{section}

\usepackage{amssymb}

\newcommand{\abs}[1]{\lvert#1\rvert}

\newcommand{\norm}[1]{\lVert#1\rVert}
\newcommand{\R}{\mathbb{R}}

\newcommand{\N}{\mathbb{N}}

\newcommand{\dif}{\,\mathrm{d}}

\usepackage{esint}

\title[Sobolev--Morrey--Campanato interpolation inequalities]{Interpolation inequalities between Sobolev and Morrey--Campanato spaces:\\
A common gateway to concentration-compactness and Gagliardo-Nirenberg interpolation inequalities}
\author{Jean Van Schaftingen}

\address{Universit\'e catholique de Louvain\\
Institut de Recherche en Math\'ematique et Physique (IRMP)\\
Chemin du Cyclotron 2 bte L7.01.01\\
1348 Louvain-la-Neuve\\
Belgium}
\email{Jean.VanSchaftingen@uclouvain.be}
\date{\today}



\begin{document}

\begin{abstract}
We prove interpolation estimates between Morrey--Campanato spaces and Sobolev spaces.
These estimates give in particular concentration-compactness inequalities in the translation-invariant and in the translation- and dilation-invariant case.
They also give in particular interpolation estimates between Sobolev spaces and functions of bounded mean oscillation.
The proofs rely on Sobolev integral representation formulae and maximal function theory.
Fractional Sobolev spaces are also covered.
\end{abstract}

\maketitle

\section{Introduction}

The subcritical Sobolev embedding states that if \(1 \le p \le q\) and \(\frac{1}{q} > \frac{1}{p} - \frac{1}{N}\), then for every \(u\) in the Sobolev space \(W^{1, p} (\R^N)\), \(u \in L^q (\R^N)\) and 
\begin{equation}
\label{ineqSobolev}
  \Bigl(\int_{\R^N} \abs{u}^q \Bigr)^\frac{p}{q} \le C \int_{\R^N} \abs{D u}^p + \abs{u}^p.
\end{equation}
Because the norms in \(W^{1, p} (\R^N)\) and \(L^q (\R^N)\) are \emph{invariant under translation}, this embedding is not compact, that is, bounded sets in \(W^{1, p} (\R^N)\) need not be precompact in \(L^q (\R^N)\).

This noncompactness is an obstacle to prove that the optimal constant in the estimate \eqref{ineqSobolev} is achieved.
One of the key observations in the concentration-compactness method of P.-L.\thinspace Lions which allows to overcome this problem \cite{Lions1984CC2} is that the elements of any  bounded sequence that does not converge to \(0\) in \(L^q (\R^N)\) can be translated in space so that the sequence of translations does not converge to \(0\) in \(L^q_{\mathrm{loc}} (\R^N)\). This fact can be deduced from the inequality  \citelist{\cite{Lions1984CC2}*{lemma I.1}\cite{Willem1996}*{lemma 1.21}\cite{MorozVanSchaftingen2013Ground}*{lemma 2.3}}:
\begin{equation}
\label{ineqLions}
  \int_{\R^N} \abs{u}^q \le C \Bigl(\sup_{x \in \R^N} \int_{B_1 (x)} \abs{u}^q\Bigr)^{1 - \frac{p}{q}} \int_{\R^N} \abs{D u}^p + \abs{u}^p.
\end{equation}

When \(p \in (1, N)\) and \(q = \frac{N p}{N - p}\), the limiting inequality for \eqref{ineqSobolev} is the critical Sobolev inequality
\begin{equation}
\label{ineqSobolevCritical}
  \Bigl(\int_{\R^N} \abs{u}^\frac{N p}{N - p} \Bigr)^{1 - \frac{p}{N}} \le C \int_{\R^N} \abs{D u}^p.
\end{equation}
This inequality is now \emph{invariant under both translations and dilations}. In particular, there are bounded sequences that do not converge to \(0\) in \(L^{N p / (N - p)} (\R^N)\) and every translation of which also converges to \(0\) in \(L^q_{\mathrm{loc}} (\R^N)\). However, for every bounded sequence \((u_n)_{n \in \N}\) in \(W^{1, p} (\R^N)\), that does not converge to \(0\) in \(L^{N p / (N - p)} (\R^N)\), there exist sequences \((x_n)_{n \in \N}\) in \(\R^N\) and \((r_n)_{n \in \N}\) in \((0, \infty)\) such that if 
\[
  v_n (y) = r_n^{(N - p)/p} u_n (x_n + r_n y),
\]
then the sequence \((v_n)_{n \in \N}\) does not converge to \(0\) in \(L^p_{\mathrm{loc}} (\R^N)\) \cite{Lions1985CC1}.
This follows from the inequality for every \(u \in W^{1, p} (\R^N)\)
\begin{equation}
\label{ineqLionsDilation}
 \int_{\R^N} \abs{u}^\frac{N p}{N - p} \le C \Biggl(\sup_{\substack{x \in \R^N\\ r > 0}} \frac{1}{r^p} \int_{B_r (x)} \abs{u}^p \Biggr)^\frac{p}{N - p} \int_{\R^N} \abs{D u}^p,
\end{equation}
which follows by H\"older's inequality from the interpolation estimate \cite{PatalucciPisante}*{theorem 1.2}
\begin{equation}
\label{ineqLionsDilationMorrey}
 \int_{\R^N} \abs{u}^\frac{N p}{N - p} \le C \norm{u}_{\mathcal{M}^{1, (N - p)/p}(\R^N)}^{p^2/(N - p)} \int_{\R^N} \abs{D u}^p, 
\end{equation}
where the Morrey norm is defined by 
\[
  \norm{u}_{\mathcal{M}^{q, \lambda}(\R^N)}
  = \sup_{\substack{x \in \R^N\\ r > 0}}  r^\lambda \Bigl( \fint_{B_r (x)} \abs{u}^q\Bigr)^\frac{1}{q}.
\]

The inequalities \eqref{ineqLions} and \eqref{ineqLionsDilation} seem at first hand quite different: the first is translation-invariant whereas the second is dilation- and translation-invariant. A first question that we address in this paper is to determine the relationship between the inequalities \eqref{ineqLions} and \eqref{ineqLionsDilation}. 
We answer this question by proving a family of inequalities of which both \eqref{ineqLions} and \eqref{ineqLionsDilation} are direct consequences:
if \(q > p > 1\) and if \(\lambda \in [0, p/(q - p))\), then for every \(u \in W^{1, p} (\R^N) \cap \mathcal{M}^{1, \lambda}_\rho(\R^N)\),
\begin{equation}
\label{ineqParticularInterpolation}
  \int_{\R^N} \abs{u}^q \le C\bigl(\rho^{-\lambda} \norm{u}_{\mathcal{M}^{1, \lambda}_\rho(\R^N)}\bigr)^{q - p} \int_{\R^N} \rho^{p} \abs{D u}^p + \abs{u}^p.
\end{equation}
where the localized Morrey norm is defined as \cite{TransiricoTroisiVitolo1995} (see also \citelist{\cite{CanaleDiGironimoVitolo1998}\cite{CasoDAmbrosioMonsurro2010}})
\[ 
\norm{u}_{\mathcal{M}^{q, \lambda}_\rho(\R^N)}
  = \sup_{\substack{x \in \R^N\\ r \in (0, \rho)}} r^\lambda \Bigl( \fint_{B_r (x)} \abs{u}^q\Bigr)^\frac{1}{q};
\]
the estimate \eqref{ineqLions} follows by the classical H\"older inequality from \eqref{ineqParticularInterpolation} with \(\lambda = \frac{N}{q}\) and \(\rho = 1\), whereas \eqref{ineqLionsDilationMorrey} is obtained by letting \(\rho \to \infty\) with \(\lambda = \frac{N - p}{p}\).

Our proof of \eqref{ineqParticularInterpolation} is based on pointwise integral estimates of a function and the maximal function theorem. It covers higher-order derivatives (theorem~\ref{theoremWkpfunction}) and fractional derivatives (theorem~\ref{theoremWspfunction}).
Our proof also provides an independent proof of the classical Sobolev and Gagliardo--Nirenberg inequalities.

Finally, we would like to mention that the statements of theorems~\ref{theoremWkpfunction} and \ref{theoremWspfunction} allow to prove an interpolation inequality between Sobolev spaces and functions of bounded mean oscillation: assuming that \(s > \ell \in \N\), if \(s \not \in \N\),  and \(p \ge 1\), then for every \(u \in W^{s, p} (\R^N) \cap \mathrm{BMO} (\R^N)\), one has \(u \in W^{\ell, \frac{sp}{\ell}} (\R^N)\) and 
\begin{equation}
  \norm{D^\ell u}_{L^{s p/\ell} (\R^N)} \le C \abs{u}_{\mathrm{BMO} (\R^N)}^{1 - \frac{s}{\ell}} \abs{u}_{W^{s, p}(\R^N)}^\frac{s}{\ell}.
\end{equation}
and if \(p > 1\) and \(s \in \N\), then for every \(u \in W^{s, p} (\R^N) \cap \mathrm{BMO} (\R^N)\), one has \(D u \in L^{\frac{sp}{\ell}} (\R^N)\) and 
\begin{equation}
  \norm{D^\ell u}_{L^{s p/\ell} (\R^N)} \le C \abs{u}_{\mathrm{BMO} (\R^N)}^{1 - \frac{s}{\ell}} \norm{D^s u}_{L^p (\R^N)}^\frac{s}{\ell}.
\end{equation}
These inequalities were known when \(s \in \N\) or \(p = 2\) \citelist{\cite{Strzelecki2006}\cite{KozonoWadade2008}\cite{MeyerRiviere2003}}.

\section{Statement of the result}

In order to state our results we recall the definition of the Campanato semi-norm\footnote{We warn the reader of the variety of conventions for the parameters in the definition of the Campanato seminorm.} \citelist{\cite{Campanato1964}\cite{RafeiroSamkoSamko2013}}
\[
  \abs{u}_{\mathcal{L}^{q, \lambda}_k}^q
  = \sup_{\substack{x \in \R^N\\ r > 0}} r^\lambda \inf_{P \in \mathcal{P}_{k - 1} (\R^N)} \fint_{B_r (x)} \abs{u - P}^q,
\]
where \(\mathcal{P}_{k - 1} (\R^N)\) denotes the space of polynomials on \(\R^N\) of degree at most \(k - 1\).
We define the localized Campanato semi-norm
\[
  \abs{u}_{\mathcal{L}^{q, \lambda}_{k, \rho} (\R^N)}^q
  = \sup_{\substack{x \in \R^N\\ r \in (0, \rho)}} r^\lambda \inf_{P \in \mathcal{P}_{k - 1} (\R^N)} \fint_{B_r (x)} \abs{u - P}^q.
\]
If we use the convention that \(\mathcal{P}_{-1} (\R^N) = \{0\}\), then we observe that 
\[
 \abs{u}_{\mathcal{L}^{q, \lambda}_{0, \rho} (\R^N)} = \norm{u}_{\mathcal{M}^{q, \lambda}_{0, \rho} (\R^N)}.
\]

It is clear from the definition that if \(\ell \le k\), then for every \(u \in \mathcal{L}^{q, \lambda}_{\ell, \rho} (\R^N)\),
\[
  \abs{u}_{\mathcal{L}^{q, \lambda}_{k, \rho} (\R^N)}^q \le \abs{u}_{\mathcal{L}^{q, \lambda}_{\ell, \rho} (\R^N)}^q;
\]
conversely \cite{Campanato1964}*{teorema 6.2},
\[
 \abs{u}_{\mathcal{L}^{q, \lambda}_{\ell, \rho} (\R^N)}
 \le C \biggl(\abs{u}_{\mathcal{L}^{q, \lambda}_{k, \rho} (\R^N)}
 + \sup_{\substack{x \in \R^N}} \rho^\lambda \inf_{P \in \mathcal{P}_{\ell - 1} (\R^N)} \fint_{B_\rho (x)} \abs{u - P}^q\biggr);
\]
that is, we only need to look at differences with low-degree polynomials only at the scale \(\rho\).

We now state our main interpolation estimate.

\begin{theorem}[Interpolation estimate]
\label{theoremWkpfunction}
Let \(N \in \N\), \(k \in \N_*\) and \(\ell \in \{0, \dotsc, k - 1\}\), \(1 < p < q <\infty\) and \(-\ell \le \lambda \le \frac{k p - \ell q}{q - p}\).
There exists a constant \(C\) such that for every \(\rho > 0\), if \(u \in W^{k, p} (\R^N) \cap \mathcal{L}^{1, \lambda}_{\ell, \rho} (\R^N)\), then \(D^\ell u \in L^q (\R^N)\) and 
\[
 \int_{\R^N} \rho^{\ell q} \abs{D^\ell u}^q \le C \bigl(\rho^{-\lambda} \abs{u}_{\mathcal{L}^{1, \lambda}_{\ell, \rho} (\R^N)}\bigr)^{q - p} \int_{\R^N} \bigl(\rho^{k p} \abs{D^k u}^p + \rho^{\ell p}\abs{D^\ell u}^p\bigr).
\]
\end{theorem}

We first discuss the relationship between the estimate of theorem~\ref{theoremWkpfunction} and similar estimates. 
If \(r \ge 1\), by the definition of the inhomogeneous Campanato space \(\mathcal{L}^{r, \lambda}_{\ell, \rho} (\R^N)\) and by the classical H\"older inequality, for every \(u \in W^{k, p} (\R^N) \cap \mathcal{L}^{1, \lambda}_{\ell, \rho} (\R^N)\), the inequality can be weakened to
\begin{equation}
 \int_{\R^N} \rho^{\ell q} \abs{D^\ell u}^q \le C \bigl(\rho^{-\lambda} \abs{u}_{\mathcal{L}^{r, \lambda}_{\ell, \rho} (\R^N)}\bigr)^{q - p} \int_{\R^N} \bigl(\rho^{k p} \abs{D^k u}^p + \rho^{\ell p} \abs{D^\ell u}^p\bigr).
\end{equation}
If \(\lambda \le N\), by the classical H\"older inequality and by monotonicity of the integral
\begin{equation}
 \abs{u}_{\mathcal{L}^{1, \lambda}_{\ell,\rho}}
 \le \sup_{x \in \R^N} \rho^{\lambda} \Bigl(\fint_{B_\rho (x)} \abs{u}^\frac{N}{\lambda}\Bigr)^\frac{\lambda}{N},
\end{equation}
so that theorem~\ref{theoremWkpfunction} gives in particular the estimate in the case where \(\ell q < k p\) and \(\frac{N}{\lambda} \ge \max(N (\frac{q - p}{k p - \ell q}), 1)\)
\begin{equation}
\label{eqLocalizedSobolev}
 \int_{\R^N} \rho^{\ell q} \abs{D^\ell u}^q \le C \Bigl(\sup_{x \in \R^N}  \fint_{B_\rho (x)} \abs{u}^\frac{N}{\lambda}\Bigr)^{(q - p)\frac{\lambda}{N}} \int_{\R^N} \bigl(\rho^{k p} \abs{D^k u}^p + \rho^{\ell p} \abs{D^\ell u}^p\bigr).
\end{equation}
In particular, if \(\frac{1}{p}-\frac{k}{N}\le \frac{1}{q} \le \frac{1}{p}\) and if we set \(\lambda = \frac{N}{q}\), we obtain the inequality
\begin{equation}
\label{ineqLionsHigher}
 \int_{\R^N} \rho^{\ell q} \abs{u}^q \le C \Bigl(\sup_{x \in \R^N}  \fint_{B_\rho (x)} \abs{u}^q\Bigr)^{1 - \frac{p}{q}} \int_{\R^N} \bigl(\rho^{k p} \abs{D^k u}^p + \rho^{\ell p} \abs{u}^p\bigr).
\end{equation}
This inequality yields \eqref{ineqLions} in particular; the estimate \eqref{ineqLionsHigher} can be proved in the wider range \(p \ge 1\) by the Gagliardo--Nirenberg interpolation inequality applied to balls and then by integration over balls; this argument is well-known for \(k = 1\) and \(p \ge 1\) \citelist{\cite{Lions1984CC2}*{lemma I.1}\cite{Willem1996}*{lemma 1.21}\cite{MorozVanSchaftingen2013Ground}*{lemma 2.3}}.

The inequality \eqref{eqLocalizedSobolev} implies a subscale of the Gagliardo--Nirenberg interpolation inequalities\citelist{\cite{Gagliardo1958}\cite{Nirenberg1959}}: if \(\ell q \le k p\) and \(t = \frac{N}{\lambda} \ge \max(N\frac{q - p}{k p - \ell q},1)\), then 
\[
  \int_{\R^N} \rho^{\ell q} \abs{D^\ell u}^q \le C \Bigl(\frac{1}{\rho^N} \int_{\R^N} \abs{u}^t \Bigr)^{\frac{q - p}{t}} \int_{\R^N}  \bigl(\rho^{k p} \abs{D^k u}^p + \rho^{\ell p} \abs{D^\ell u}^p\bigr).
\]
We have in particular, if \(\frac{1}{p} - \frac{k}{N} \le \frac{1}{q} \le \frac{1}{p}\), the classical Sobolev inequality
\[
  \frac{1}{\rho^N} \int_{\R^N} \abs{u}^q \le C \Bigl(\frac{1}{\rho^N} \int_{\R^N} \bigl(\rho^{k p} \abs{D^k u}^p + \abs{u}^p\bigr)\Bigr)^\frac{q}{p}.
\]
If \(\ell q \le k p\) and \(\lambda = 0\), we have the interpolation inequality
\[
 \int_{\R^N} \rho^{\ell q} \abs{D^\ell u}^q \le C \norm{u}_{L^\infty (\R^N)}^{q - p} \int_{\R^N} \bigl(\rho^{k p} \abs{D^k u}^p + \rho^{\ell p} \abs{D^\ell u}^p\bigr).
\]
If \(\ell \in \{1, \dotsc, k - 1\}\) and \(\lambda = 0\), the latter inequality can be improved by the isomorphism between Campanato spaces and functions of bounded mean oscillation (BMO) \citelist{\cite{Campanato1964}*{p. 159}\cite{RafeiroSamkoSamko2013}*{theorem 4.3}\cite{Peetre1969}}, the estimate of theorem~\ref{theoremWkpfunction}: if \(\ell q \le kp\), then 
\begin{equation}
\label{equationnonhomogeneousGagliardoNirenberg}
 \int_{\R^N} \rho^{\ell q} \abs{D^\ell u}^q \le C \abs{u}_{\mathrm{BMO}_\rho (\R^N)}^{q - p} \int_{\R^N} \bigl(\rho^{k p} \abs{D^k u}^p + \rho^{\ell p} \abs{D^\ell u}^p\bigr)
\end{equation}
where the local bounded mean oscillation seminorm is defined by 
\[
  \abs{u}_{\mathrm{BMO}_\rho (\R^N)} = \sup_{\substack{x \in \R^N\\ 0 < r < \rho}} \fint_{B_r (x)}\fint_{B_r (x)} \abs{u (y) - u (z)}\dif z \dif y.
\]
We remark also, that when \(\lambda \in (-\ell, 0)\) is not an integer, the inhomogeneous Campanato seminorm is a H\" older seminorm \citelist{\cite{Campanato1964}*{teorema 4.1}\cite{RafeiroSamkoSamko2013}*{theorem 4.4}\cite{JansonTaiblesonWeiss1983}}.

As the proof of theorem~\ref{theoremWkpfunction} does not depend on any Sobolev or Gagliardo--Nirenberg inequality, the proof of theorem~\ref{theoremWkpfunction} provides an alternative method to prove these inequalities based essentially on the Sobolev integral representation and the maximal function theorem.

In the homogeneous case \(\lambda = (k p -\ell q)/(q - p)\), if we let \(\rho \to \infty\), theorem~\ref{theoremWkpfunction} implies an interpolation result between Morrey spaces and Sobolev spaces
\begin{equation}
\label{equationHomogeneousMorreyInterpolation}
 \int_{\R^N} \abs{D^\ell u}^q \le C \abs{u}_{\mathcal{L}^{1, (k p - \ell q)/(q - p)}_{\ell}} (\R^N)^{q - p} 
 \int_{\R^N} \abs{D^k u}^p.
\end{equation}
In particular, when \(k p < N\) and \(\frac{1}{q} = \frac{1}{p} - \frac{k}{N}\) in \eqref{equationHomogeneousMorreyInterpolation}, we obtain the generalization of \eqref{ineqLionsDilationMorrey}
\begin{equation}
 \int_{\R^N} \abs{D^\ell u}^\frac{N p}{N - (k - \ell)p} \le C \norm{u}_{\mathcal{M}^{1, N/p - k}}^{\frac{(k - \ell) p^2}{N - (k - \ell) p} (\R^N)} 
 \int_{\R^N} \abs{D^k u}^p
\end{equation}
which was known for \(\ell = 1\) and \(p = 2\) or \(k = 2\) \citelist{\cite{PatalucciPisante}*{theorems 1.1 and 1.2}}.

If \(k p = q \ell\), then the inequality \eqref{equationHomogeneousMorreyInterpolation} becomes, by the equivalence between the Campanato space \(\mathcal{L}^{1, 0}_{\ell} (\R^N)\) and the space of functions with bounded mean oscillation \(\mathrm{BMO} (\R^N)\), 
\begin{equation}
\label{homogeneousGagliardoNirenberg}
 \int_{\R^N} \abs{D^\ell u}^\frac{k p}{\ell} \le C \abs{u}_{\mathrm{BMO} (\R^N)}^{(\frac{k}{\ell} - 1) p}
 \int_{\R^N} \abs{D^k u}^p,
\end{equation}
where the bounded mean oscillation semi-norm is defined by 
\[
  \abs{u}_{\mathrm{BMO} (\R^N)} = \sup_{\substack{x \in \R^N\\ r > 0}} \fint_{B_r (x)}\fint_{B_r (x)} \abs{u (y) - u (z)}\dif z \dif y.
\]
The estimate \eqref{homogeneousGagliardoNirenberg} is also the limit when \(\rho \to \infty\) of \eqref{equationnonhomogeneousGagliardoNirenberg}. 
This estimate was proved by embeddings in the Besov scale space \cite{MeyerRiviere2003}*{theorem 1.4} and by duality between \(\mathrm{BMO} (\R^N)\) and the real Hardy space \(\mathcal{H}^1 (\R^N)\) \cite{Strzelecki2006}*{theorem 1.2}. Similarly, when \((\ell q - k p)/(p - q)\) is positive and not an integer, we recover from \eqref{equationHomogeneousMorreyInterpolation} interpolation estimates with H\"older continuous functions \cite{Nirenberg1959}
\begin{equation} 
\int_{\R^N} \abs{D^\ell u}^q \le C \abs{u}_{C^{(\ell q - k p)/(q - p)} (\R^N)}^{q - p} 
 \int_{\R^N} \abs{D^k u}^p;
\end{equation}
the latter inequality still holds for integer \((\ell q - k p)/(q - p)\) if one takes the semi-norm in the corresponding homogeneous Zygmund space.

When \(k = 1\), the inequality \eqref{equationHomogeneousMorreyInterpolation} also follows from the stronger interpolation inequality \cite{Ledoux2003}*{theorem~1} 
\[
 \int_{\R^N} \abs{u}^q \le C \norm{u}_{\dot{B}^{- p/(q - p)}_{\infty, \infty} (\R^N)}^{q - p} 
 \int_{\R^N} \abs{D u}^p;
\]
by embeddings of the Morrey class \(\mathcal{M}^{1, p/(q - p)} (\R^N)\) in the Besov space \(B^{- p/(q - p)}_{\infty, \infty}(\R^N)\) \citelist{\cite{Ledoux2003}*{\S 2.3}\cite{PatalucciPisante}*{lemma 3.4}} (see also \cite{YuanSickelYang2010}*{corollary 3.3, proposition 2.4 and corollary 2.2})
the latter approach covers specifically the case \(p = 1\)  \citelist{\cite{CohenMeyerOru1998}\cite{CohenDeVorePetrushevXu1999}\cite{CohenDahmenDaubechiesDeVore2003}}. (For \(p = 2\) and \(q = 4\) see also \cite{MeyerRiviere2003}*{theorem 2.6}.)

If \(q > p (1 + \frac{k}{N})\), the Lorentz space \(L^{N (q - p)/(kp),\infty} (\R^N)\) is continuously embedded in \(\mathcal{M}^{1, k p/(q - p)} (\R^N)\) and thus the estimate \eqref{equationHomogeneousMorreyInterpolation} implies
\begin{equation}
\label{equationLorentz}
 \int_{\R^N} \abs{u}^q \le C \norm{u}_{L^{N (q - p)/kp,\infty} (\R^N)} ^{q - p} 
 \int_{\R^N} \abs{D^k u}^p.
\end{equation}
If \( p \in [1, \frac{N}{k})\), the inequality \eqref{equationLorentz} can also be deduced from the embedding of the Sobolev space \(W^{k, p} (\R^N)\) into the Lorentz space \(L^{ Np /(N - k p), p} (\R^N) \) \citelist{\cite{Peetre1966}*{th\'eor\`eme 7.1}\cite{Alvino1977}\cite{Tartar1998}\cite{ONeil1963}} and by interpolation between Lorentz spaces. These inequalities imply the weaker inequality \cite{Gerard1998}:
\[
 \int_{\R^N} \abs{u}^q \le C  \abs{u}_{B^k_{p, \infty} (\R^N)} ^{q - p} 
 \int_{\R^N} \abs{D^k u}^p
\]
by the homogeneous embedding of \(B^k_{p, \infty} (\R^N)\) in \(L^{p, q} (\R^N)\) which is a consequence of the embeddings of Besov spaces into Lebesgue spaces and interpolation theorems \cite{Triebel1978}.

\section{Proof of the estimate}

The proof of theorem~\ref{theoremWkpfunction} will use a pointwise estimate on the value of a function by its derivatives.

\begin{lemma}[Pointwise estimate of the value of a function]
\label{lemmaDk}
There exists a constant \(C > 0\) such that for every \(u \in W^{1, k}_\mathrm{loc} (\R^N)\), for almost every \(x \in \R^N\) and for every \(R > 0\),
\[
  R^\ell \abs{D^\ell u (x)} \le C \Bigl(\int_{B_R (0)} \frac{\abs{D^k u (y)}}{\abs{x - y}^{N - k}} \dif y + \fint_{B_R (x)} \abs{u}\Bigr) .
\]
\end{lemma}

When \(\ell = 0\) this estimate is a direct consequence of the Sobolev integral representation formula \cite{Mazya2010}*{theorem 1.1.10/1}.
It has appeared as an intermediate step of pointwise interpolation for derivatives \cite{MazyaShaposhnikova1999}*{(15)}. We provide here for the sake of completeness a complete argument following that is a combination of these proofs \citelist{\cite{Mazya2010}*{theorem 1.1.10/1}\cite{MazyaShaposhnikova1999}*{theorem 1}}.

\begin{proof}[Proof of lemma~\ref{lemmaDk}]
We fix \(\eta \in C^k_c (B_1)\) such that \(\int_{B_1} \eta = 1\) and we define for every \(x \in \R^N\) and \(w_1, \dotsc, w_k \in \R^N\), the function \(g : (0, \infty) \to \R\) for each \(r \in (0, \infty)\) by
\[
  g (r) = \int_{B_1} D^\ell u (x + r z)[w_1, \dotsc, w_\ell] \eta (z)\dif z = \int_{B_r (x)} D^\ell u (y)[w_1, \dotsc, w_\ell] \eta_r (y - x),
\]
where we have set for every \(r > 0\) and \(z \in \R^N\), \(\eta_r (z) = \eta (z /r)/r^N\).
The function \(g\) is \(k- j\) times continuously differentiable and for every \(j \in \{0, \dotsc, k - \ell\}\) and \(r \in (0, \infty)\),
\[
\begin{split}
  g^{(j)} (r) &= \int_{B_1} D^{\ell + j} u (x + rz)[w_1, \dotsc, w_\ell,z, \dotsc, z] \eta (z) \dif z\\
  &= \frac{1}{r^j} \int_{B_r (x)} D^\ell u (y)[w_1, \dotsc, w_\ell,y - x, \dotsc, y - x] \eta_r (y - x) \dif y.
\end{split}
\]
By integration by parts, for every \(j \in \{0, \dotsc, k - \ell\}\), there exists a function \(\eta^j \in C^{k - j}_c (B_1; \mathrm{Lin}^\ell (\R^N))\) such that 
\[
  \int_{B_1} D^{\ell + j} v (z)[w_1, \dotsc, w_\ell,z, \dotsc, z] \eta (z) \dif z
  = (-1)^j\int_{B_1} v (z) \eta^j (z)[w_1, \dotsc, w_\ell]\dif z,
\]
and hence
\[
\begin{split}
  g^{(j)} (r) &= \int_{B_1} D^{\ell + j} u (x + rz)[w_1, \dotsc, w_\ell,z, \dotsc, z] \eta (z) \dif z\\
  &= \frac{(-1)^j}{r^{\ell + j}} \int_{B_1} u (x + rz) \eta^j (z)[w_1, \dotsc, w_\ell]\dif z\\
  &= \frac{(-1)^j}{r^{\ell + j}} \int_{B_r (x)} u (y) \eta^j_r (y - x) [w_1, \dotsc, w_\ell] \dif y.
\end{split}
\]
where we have set for every \(r > 0\) and \(z \in \R^N\), \(\eta^j_r (z) = \eta^j (z/r)/r^N\).

If \(x\) is a Lebesgue point of the function \(D^\ell u\), then 
\[
  \lim_{r \to 0} g (r) = u (x).
\]
Moreover, since \(u \in W^{1, k}_{\mathrm{loc}}(\R^N)\), for almost every \(x \in \R^N\) and for every \(R > 0\),
\[
  \int_{B_R (a)} \int_{B_{R} (x)} \frac{\abs{D^k u (y)}}{\abs{x - y}^{N - k}}\dif y \dif x
  \le \Bigl(\int_{B_{2 R} (a)} \abs{D^k u}\Bigr) \Bigl(\int_{B_{R}} \frac{1}{\abs{z}^{N - k}} \dif z \Bigr) < \infty,
\]
hence for almost every \(x \in \R^N\), 
\[
\int_{B_{R} (x)} \frac{\abs{D^k u (y)}}{\abs{x - y}^{N - k}}\dif y < \infty. 
\]
By the integral version of the Taylor expansion of \(g\) at \(R\), we write
\begin{multline}
\label{eqTaylor}
  D^\ell u (x)[w_1, \dotsc, w_\ell] = \lim_{r \to 0} g (r) = \sum_{j = 0}^{k -\ell - 1} \frac{g^{(j)} (R)\, (-R)^j}{j !}
  - \int_0^R \frac{g^{(k)} (r)\, (-r)^{k - 1}}{(k - 1) !}\dif r\\
  \shoveleft{= \sum_{j = 0}^{k - \ell - 1} \frac{1}{R^\ell} \int_{B_R (x)} u (y) \eta_r^{(j)} (y - x)[w_1, \dotsc, w_\ell] \dif y}\\
  + \int_0^R  \int_{B_r (x)} D^k u (y)[w_1, \dotsc, w_\ell, x - y, \dotsc, x - y] \eta_r (y - x) \dif y \frac{1}{(k - 1) ! r}\dif r .
\end{multline}
By Fubini's theorem, 
\begin{multline*}
\int_0^R  \int_{B_r (x)} D^k u (y)[w_1, \dotsc, w_\ell, x - y, \dotsc, x - y] \eta_r (y - x) \dif y \frac{1}{(k - 1) ! r}\dif r \\
= \frac{1}{(k - 1)!} \int_{B_r (x)} D^k u (y)[w_1, \dotsc, w_\ell, x - y, \dotsc, x - y] H_r (y - x) \dif y \dif r,
\end{multline*}
where we have set for \(r > 0\) and \(z \in b\) \(H_r (z) = H (z / r)/r^N\) and 
\[
  H (z) = \int_{\abs{z}}^\infty \frac{\eta (t z)}{t^N} \dif t.
\]
Since \(\abs{H_r (z)} \le C \abs{z}^{-N}\) and \(\abs{\eta^j_r (x)} \le Cr^{-N}\), we conclude that  
\[
  R^\ell \abs{D^\ell u (x)} \le C \Bigl(\fint_{B_r (x)} \abs{u} 
  + \int_{B_R (x)} \frac{\abs{D^k u (y)}}{\abs{x - y}^{N - k}}\dif y \Bigr).\qedhere
\]
\end{proof}

\begin{proof}[Proof of theorem~\ref{theoremWkpfunction}]
For almost every \(x \in \R^N\), for every \(R > 0\) and every \(P \in \mathcal{P}_{\ell - 1} (\R^N)\), we bound by the pointwise estimate  (lemma~\ref{lemmaDk}), since \(D^\ell P = 0\) on \(\R^N\),
\[
\begin{split}
 R^\ell \abs{D^\ell u (x)} &= R^\ell \abs{D^\ell (u - P) (x)}\\
 &\le C \Bigl(\int_{B_R (0)} \frac{\abs{D^k (u - P) (y)}}{\abs{x - y}^{N - k}} \dif y + \fint_{B_R (x)} \abs{u - P}\Bigr)\\
 &= C \Bigl(\int_{B_R (0)} \frac{\abs{D^k u (y)}}{\abs{x - y}^{N - k}} \dif y + \fint_{B_R (x)} \abs{u - P}\Bigr).
\end{split}
\]
We observe that by Fubini's theorem
\[
  \int_{B_R (x)} \frac{\abs{D^k u (y)}}{\abs{x - y}^{N - k}} \dif y
  = (N - k)\int_0^R \Bigl(\frac{1}{r^N} \int_{B_r (x)} \abs{D^k u}\Bigr) r^{k - 1}\dif r
  + \frac{1}{R^{N - k}}\int_{B_R (x)} \abs{D^k u}.
\]
We fix \(\beta > 0\).
Hence, if \(R \le \rho\), in view of our previous computation and by definition of the maximal function and the Morrey-Campanato norm, 
\begin{equation}
\label{eqFirstBound}
\rho^\ell \abs{D^\ell u (x)} \le C \rho^{\ell + \beta} \bigl( R^{k-\ell-\beta} \mathcal{M}(\abs{D^k u}) (x) + R^{- \lambda -\ell - \beta} \abs{u}_{\mathcal{L}^{1, \lambda}_{\ell, \rho} (\R^N)}\bigr).
\end{equation}
If \(\abs{u}_{\mathcal{L}^{1, \lambda}_{\ell, \rho} (\R^N)}\le \mathcal{M}(\abs{D^k u}) (x) \rho^{k + \lambda}\), we take
\[
 R = \biggl(\frac{\abs{u}_{\mathcal{L}^{1, \lambda}_{\ell, \rho} (\R^N)}}{ \mathcal{M}(\abs{D^k u}) (x)} \biggr)^\frac{1}{k + \lambda}.
\]
and we obtain
\begin{equation}\label{eqCase1}
\begin{split}
 \rho^\ell\abs{D^\ell u (x)} &\le C \rho^{\ell + \beta} \bigl( \mathcal{M}(\abs{D^k u}) (x)\bigr)^\frac{\lambda + \ell + \beta}{\lambda + k} \abs{u}_{\mathcal{L}^{1, \lambda}_{\ell, \rho} (\R^N)}^\frac{k - \ell - \beta}{k + \lambda}\\
 &\le C \bigl(\rho^k  \mathcal{M}(\abs{D^k u}) (x)\bigr)^\frac{\lambda + \ell + \beta}{\lambda + k} \bigl(\rho^{-\lambda} \abs{u}_{\mathcal{L}^{1, \lambda}_{\ell, \rho} (\R^N)}\bigr)^\frac{k - \ell - \beta}{k + \lambda}.
\end{split}
\end{equation}
Otherwise, we observe that by \eqref{eqFirstBound} with \(R = \rho\),
\[
 \rho^\ell \abs{D^\ell u (x)} 
 \le C \rho^{-\lambda} \abs{u}_{\mathcal{L}^{1, \lambda}_{\ell, \rho} (\R^N)}
\]
and thus if \(- \lambda -\ell \le \beta \le k - \ell\)
\begin{equation}
\label{eqCase2}
\rho^\ell \abs{D^\ell u (x)} \le C \bigl(\rho^\ell \abs{D^\ell u (x)}\bigr)^\frac{\lambda + \ell + \beta}{\lambda + k} \bigl(\rho^{-\lambda} \abs{u}_{\mathcal{L}^{1, \lambda}_{\ell, \rho} (\R^N)}\bigr)^\frac{k - \ell -\beta}{\lambda + k}.
\end{equation}

Hence, we have in both cases in view of \eqref{eqCase1} and \eqref{eqCase2}
\begin{equation}
\label{equationInterpolationLocal}
 \rho^\ell \abs{D^\ell u (x)} \le C \bigl(\rho^k \mathcal{M}(\abs{D^k u}) (x) + \rho^\ell \abs{D^\ell u (x)}\bigr)^\frac{\lambda + \ell + \beta}{\lambda + k} \bigl(\rho^{-\lambda} \abs{u}_{\mathcal{L}^{1, \lambda}_{\ell, \rho} (\R^N)}\bigr)^\frac{k - \ell -\beta}{\lambda + k}.
\end{equation}
If we take \(\beta = \frac{p (k + \lambda)}{q} - \lambda - \ell\), we observe that by the assumption \( \lambda \le \frac{k p - \ell q}{q - p}\), \(\beta \ge 0\) and thus, since \(\lambda \ge - \ell\), \(\beta \ge -\lambda - \ell\). Moreover \(p \le q\) implies that \(\beta \le k - \ell\). We obtain thus 
\[
 \rho^{q \ell} \int_{\R^N} \abs{D^\ell u}^q \le C \bigl(\rho^{-\lambda} \abs{u}_{\mathcal{L}^{1, \lambda}_{\ell, \rho} (\R^N)}\bigr)^{q - p} \int_{\R^N} \bigl(\rho^k \mathcal{M} (\abs{D^k u}) + \rho^{\ell} \abs{D^\ell u}\bigr)^p 
\]
By the maximal function theorem \cite{Stein1970}*{theorem I.1}, we deduce the desired estimate.
\end{proof}

\begin{remark}
The estimate \eqref{equationInterpolationLocal} is a variant of the local pointwise interpolation estimate by maximal functions \cite{MazyaShaposhnikova1999}*{remark 3}
\[
  \rho^\ell \abs{D^\ell u (x)} \le \bigl(\mathcal{M}_\rho (\abs{D^k u})(x) + \rho^{-\ell} \mathcal{M}_\rho \abs{u} (x)\bigr)^\frac{\ell}{k} \bigl(\mathcal{M}_\rho \abs{u} (x) \bigr)^{1 - \frac{\ell}{k}},
\]
where the localized maximal function operator is defined by \(\mathcal{M}_\rho (f) = \sup_{0 < r < \rho} \fint_{B_r (x)} \abs{f}\).
\end{remark}

\section{The fractional case}

In this section we study a fractional counterpart of theorem~\ref{theoremWspfunction}.

\begin{theorem}[Interpolation estimate of the function]
\label{theoremWspfunction}
Let \(N \in \N\), \(k \in \N_*\) and \(\ell \in \{0, \dotsc, k - 1\}\), \(1 \le p < q <\infty\), \(0 < \sigma < 1\) and \(-\ell \le \lambda \le \frac{(k + \sigma) p - \ell q}{q - p}\).
There exists a constant \(C\) such that for every \(\rho > 0\), if \(u \in W^{k + \sigma, p} (\R^N) \cap \mathcal{L}^{1, \lambda}_{\ell, \rho} (\R^N)\), then \(D^\ell u \in L^q (\R^N)\) and 
\begin{multline*}
 \int_{\R^N} \rho^{\ell q} \abs{D^\ell u}^q\\
 \le C \bigl(\rho^{-\lambda} \abs{u}_{\mathcal{L}^{1, \lambda}_{\ell, \rho} (\R^N)}\bigr)^{q - p} \Bigl(\rho^{p(k + \sigma)} \int_{\R^N} \int_{\R^N} \frac{\abs{D^k u (x) - D^k u (y)}^p}{\abs{x - y}^{N + \sigma p}} \dif x \dif y + \int_{\R^N} \rho^{\ell p} \abs{D^\ell u}^p\Bigr).
\end{multline*}
\end{theorem}

In contrast with theorem~\ref{theoremWkpfunction}, the case \(p = 1\) is covered.
Theorem~\ref{theoremWspfunction} has the same consequences as its counterpart theorem~\ref{theoremWkpfunction}. We mention here some of the most striking consequences.

In the homogeneous case \(\lambda = \frac{(k + \sigma) p - \ell q}{q - p}\), we obtain the fractional counterpart of \eqref{equationHomogeneousMorreyInterpolation}: if 
\begin{equation}
\label{equationHomogeneousMorreyInterpolationFractional}
  \int_{\R^N} \abs{D^\ell u}^q \le C \bigl(\abs{u}_{\mathcal{L}^{1, ((k + \sigma) p - \ell q)/(q - p)}_{\ell, \rho} (\R^N)}\bigr)^{q - p} \int_{\R^N} \frac{\abs{D^k u (x) - D^k u (y)}^p}{\abs{x - y}^{N + sp}} \dif x \dif y.
\end{equation}
In particular, if \(p (k + \sigma) < N\), then 
\begin{equation}
\label{equationHomogeneousMorreyInterpolationFractionalCritical}
  \int_{\R^N} \abs{u}^\frac{N p}{N - (k + \sigma) p} \le C \bigl(\abs{u}_{\mathcal{L}^{1, N/p - (k + \sigma)}_{\ell, \rho} (\R^N)}\bigr)^\frac{(k + \sigma)p^2}{N - (k + \sigma)p} \int_{\R^N} \frac{\abs{D^k u (x) - D^k u (y)}^p}{\abs{x - y}^{N + sp}} \dif x \dif y.
\end{equation}
The estimate \eqref{equationHomogeneousMorreyInterpolationFractional} was known for \(p = 2\) \cite{PatalucciPisante}*{theorem 1.1}.

We also have the interpolation inequality for \(\ell \ge 1\),
\begin{equation}
  \int_{\R^N} \abs{D^\ell u}^\frac{p (k + \sigma)}{\ell} \le C \bigl(\abs{u}_{\mathrm{BMO} (\R^N)}\bigr)^{(\frac{k + \sigma}{\ell} - 1) p} \int_{\R^N} \frac{\abs{D^k u (x) - D^k u (y)}^p}{\abs{x - y}^{N + sp}} \dif x \dif y;
\end{equation}
this inequality is a consequence of interpolation inequalities between Besov spaces \cite{MachiharaOzawa2003}.

\begin{lemma}
\label{lemmaDkFract}
There exists a constant \(C > 0\) such that for every \(u \in W^{1, k}_\mathrm{loc} (\R^N)\), for every \(x \in \R^N\) and every \(R > 0\),
\[
  \abs{D^\ell u (x)} \le C \Bigl(\int_{B_R (x)} \frac{\abs{D^k u (y)- D^k u (x)}}{\abs{x - y}^{N - k}} \dif y + \fint_{B_R (x)} \abs{u}\Bigr) .
\]
\end{lemma}

This inequality implies by H\"older's inequality the fractional interpolation estimate \citelist{\cite{MazyaShaposhnikova1999}*{(32)}\cite{MazyaShaposhnikova2002}*{appendix}}
\[
  \abs{D^\ell u (x)} \le C \bigl(\mathcal{M} \abs{u} (x)\bigr)^{1 - \frac{\ell}{k + \sigma}} \Bigl(\int_{\R^N} \frac{\abs{D^k u (y)- D^k u (x)}^p}{\abs{x - y}^{N + \sigma p}}\Bigr)^\frac{\ell}{p (k + \sigma)};
\]
the inequality of the lemma appears in fact in the proof of the latter inequality \cite{MazyaShaposhnikova1999}*{(32)}. We give a proof of the lemma for the sake of completeness.

\begin{proof}[Proof of lemma~\ref{lemmaDkFract}]
The proof begins as the proof of lemma~\ref{lemmaDk}. Instead of \eqref{eqTaylor}, we write
\[
  D^\ell u (x)[w_1, \dotsc, w_\ell] = \lim_{r \to 0} g (r) = \sum_{j = 0}^{k -\ell} \frac{g^{(j)} (R)\, (-R)^j}{j !}
  - \int_0^R \frac{g^{(k)} (r) - g^{(k)} (R) }{(k - 1) !}(-r)^{k - 1}\dif r.\\
\]
We first have as previously
\[
  \sum_{j = 0}^{k -\ell} \frac{g^{(j)} (R)\, (-R)^j}{j !} = \sum_{j = 0}^{k - \ell} \frac{1}{R^\ell} \int_{B_R (x)} u (y) \eta_r^{(j)} (y - x)[w_1, \dotsc, w_k] \dif y.  
\]
Next, we have 
\begin{multline*}
  \int_0^R \frac{1}{r^k} \frac{g^{(k)} (r) - g^{(k)} (R) }{(k - 1) !}(-r)^{k - 1}\dif r\\
  \shoveleft{=-\int_0^R  \int_{B_1} D^k u (x + r z)[w_1, \dotsc, w_k, z, \dotsc, z]} - D^k u (x)[w_1, \dotsc, w_k, z, \dotsc, z]\\
  + D^k u (x)[w_1, \dotsc, w_k, z, \dotsc, z] - D^k u (x + R z)[w_1, \dotsc, w_k, z, \dotsc, z] )\\ \eta_r (y - x) \dif y \frac{(-r)^{k - \ell - 1}}{(k - \ell - 1) !}\dif r.
\end{multline*}
and we conclude by changes of variable and Fubini's theorem.
\end{proof}

\begin{proof}[Proof of theorem~\ref{theoremWspfunction}]
For almost every \(x \in \R^N\), for every \(R > 0\) and every \(P \in \mathcal{P}_{\ell - 1} (\R^N)\), we bound by the pointwise estimate  (lemma~\ref{lemmaDk}), since \(D^\ell P = 0\),
\[
 R^\ell \abs{D^\ell u (x)} \le C \Bigl(\int_{B_R (0)} \frac{\abs{D^k u (y) - D^k u (x)}}{\abs{x - y}^{N - k}} \dif y + \fint_{B_R (x)} \abs{u - P}\Bigr).
\]
We fix \(\beta > 0\). By H\"older's inequality and by definition of the Campanato norm, if \(R \le \rho\),
\begin{equation}
\label{eqFirstBoundFractional}
\rho^\ell \abs{D^\ell u (x)} \le C \rho^{\ell + \beta} \bigl( R^{k + \sigma -\ell-\beta} D_{\sigma, p} (D^k u) (x) + R^{- \lambda -\ell - \beta} \abs{u}_{\mathcal{L}^{1, \lambda}_{\ell, \rho} (\R^N)}\bigr),
\end{equation}
where we use the notation
\[
  D_{\sigma, p} (D^k u) (x) = \Bigl(\int_{\R^N} \frac{\abs{D^k u (y) - D^k u (x)}}{\abs{x - y}^{N + \sigma p}} \Bigr)^\frac{1}{p}.
\]
If \(\abs{u}_{\mathcal{L}^{1, \lambda}_{\ell, \rho} (\R^N)}\le D_{\sigma, p} (D^k u) (x) \rho^{k + \sigma + \lambda}\), we take
\[
 R = \biggl(\frac{\abs{u}_{\mathcal{L}^{1, \lambda}_{\ell, \rho} (\R^N)}}{D_{\sigma, p} (D^k u) (x)} \biggr)^\frac{1}{k + \sigma + \lambda}.
\]
and we obtain
\begin{equation}\label{eqCase1Fract}
\begin{split}
 \rho^\ell\abs{D^\ell u (x)} &\le C \rho^{\ell + \beta} \bigl(D_{\sigma, p} (D^k u) (x)\bigr)^\frac{\lambda + \ell + \beta}{\lambda + k + \sigma} \abs{u}_{\mathcal{L}^{1, \lambda}_{\ell, \rho} (\R^N)}^\frac{k - \ell - \beta}{k + \sigma + \lambda}\\
 &\le C \bigl(\rho^{k + \sigma} D_{\sigma, p} (D^k u) (x)\bigr)^\frac{\lambda + \ell + \beta}{\lambda + k + \sigma} \bigl(\rho^{-\lambda} \abs{u}_{\mathcal{L}^{1, \lambda}_{\ell, \rho} (\R^N)}\bigr)^\frac{k + \sigma - \ell - \beta}{k + \sigma + \lambda}.
\end{split}
\end{equation}
Otherwise, we observe that by \eqref{eqFirstBoundFractional} with \(R = \rho\),
\[
 \rho^\ell \abs{D^\ell u (x)} 
 \le C \rho^{-\lambda} \abs{u}_{\mathcal{L}^{1, \lambda}_{\ell, \rho} (\R^N)}
\]
and thus if \(- \lambda -\ell \le \beta \le k + \sigma - \ell\)
\begin{equation}
\label{eqCase2Fract}
\rho^\ell \abs{D^\ell u (x)} \le C \bigl(\rho^\ell \abs{D^\ell u (x)}\bigr)^\frac{\lambda + \ell + \beta}{\lambda + k + \sigma} \bigl(\rho^{-\lambda} \abs{u}_{\mathcal{L}^{1, \lambda}_{\ell, \rho} (\R^N)}\bigr)^\frac{k + \sigma - \ell -\beta}{\lambda + k + \sigma}.
\end{equation}
Hence, we have in both cases, in view of \eqref{eqCase1Fract} and \eqref{eqCase2Fract},
\begin{equation}
\label{equationInterpolationLocal}
 \rho^\ell \abs{D^\ell u (x)} \le C \bigl(\rho^{k + \sigma} D_{\sigma, p} (D^k u) (x) + \rho^\ell \abs{D^\ell u (x)}\bigr)^\frac{\lambda + \ell + \beta}{\lambda + k + \sigma} \bigl(\rho^{-\lambda} \abs{u}_{\mathcal{L}^{1, \lambda}_{\ell, \rho} (\R^N)}\bigr)^\frac{k + \sigma - \ell -\beta}{\lambda + k + \sigma}.
\end{equation}
We take \(\beta = \frac{p (k + \sigma + \lambda)}{q} - \lambda - \ell\) and we conclude with 
\[
 \rho^{q \ell} \int_{\R^N} \abs{D^\ell u}^q \le C \bigl(\rho^{-\lambda} \abs{u}_{\mathcal{L}^{1, \lambda}_{\ell, \rho} (\R^N)}\bigr)^{q - p} \int_{\R^N} \bigl(\rho^{k + \sigma} D_{\sigma, p} (D_k u) (x) + \rho^{\ell} \abs{D^\ell u}\bigr)^p. \qedhere 
\]
\end{proof}

The above proof allows to recover in particular the unpublished elementary proof of fractional Sobolev embeddings of H.\thinspace{}Brezis.


\begin{bibdiv}
\begin{biblist}

\bib{Alvino1977}{article}{
   author={Alvino, Angelo},
   title={Sulla diseguaglianza di Sobolev in spazi di Lorentz},
   journal={Boll. Un. Mat. Ital. A (5)},
   volume={14},
   date={1977},
   number={1},
   pages={148--156},
}

\bib{Campanato1964}{article}{
   author={Campanato, S.},
   title={Propriet\`a di una famiglia di spazi funzionali},
   journal={Ann. Scuola Norm. Sup. Pisa (3)},
   volume={18},
   date={1964},
   pages={137--160},
}

\bib{CanaleDiGironimoVitolo1998}{article}{
   author={Canale, Anna},
   author={Di Gironimo, Patrizia},
   author={Vitolo, Antonio},
   title={Functions with derivatives in spaces of Morrey type and elliptic
   equations in unbounded domains},
   journal={Studia Math.},
   volume={128},
   date={1998},
   number={3},
   pages={199--218},
   issn={0039-3223},
}
\bib{CasoDAmbrosioMonsurro2010}{article}{
   author={Caso, Loredana},
   author={D'Ambrosio, Roberta},
   author={Monsurr{\`o}, Sara},
   title={Some remarks on spaces of Morrey type},
   journal={Abstr. Appl. Anal.},
   volume={\vspace{0pt}},
   date={2010},
   pages={Art. ID 242079, 22},
   issn={1085-3375},
}

\bib{CohenDahmenDaubechiesDeVore2003}{article}{
   author={Cohen, Albert},
   author={Dahmen, Wolfgang},
   author={Daubechies, Ingrid},
   author={DeVore, Ronald},
   title={Harmonic analysis of the space BV},
   journal={Rev. Mat. Iberoamericana},
   volume={19},
   date={2003},
   number={1},
   pages={235--263},
   issn={0213-2230},
}

\bib{CohenDeVorePetrushevXu1999}{article}{
   author={Cohen, Albert},
   author={DeVore, Ronald},
   author={Petrushev, Pencho},
   author={Xu, Hong},
   title={Nonlinear approximation and the space $\mathrm{BV}(\R^2)$},
   journal={Amer. J. Math.},
   volume={121},
   date={1999},
   number={3},
   pages={587--628},
   issn={0002-9327},
}


\bib{CohenMeyerOru1998}{article}{
   author={Cohen, A.},
   author={Meyer, Y.},
   author={Oru, F.},
   title={Improved Sobolev embedding theorem},
   conference={
      title={S\'eminaire sur les \'Equations aux D\'eriv\'ees Partielles,
      1997--1998},
   },
   book={
      publisher={\'Ecole Polytech.},
      place={Palaiseau},
   },
   date={1998},
   note={XVI},
}

\bib{Gagliardo1958}{article}{
   author={Gagliardo, Emilio},
   title={Propriet\`a di alcune classi di funzioni in pi\`u variabili},
   journal={Ricerche Mat.},
   volume={7},
   date={1958},
   pages={102--137},
   issn={0035-5038},
}

\bib{Gerard1998}{article}{
   author={G{\'e}rard, Patrick},
   title={Description du d\'efaut de compacit\'e de l'injection de Sobolev},
   journal={ESAIM Control Optim. Calc. Var.},
   volume={3},
   date={1998},
   pages={213--233},
   issn={1292-8119},
}



\bib{JansonTaiblesonWeiss1983}{article}{
   author={Janson, Svante},
   author={Taibleson, Mitchell},
   author={Weiss, Guido},
   title={Elementary characterizations of the Morrey-Campanato spaces},
   conference={
      title={Harmonic analysis},
      address={Cortona},
      date={1982},
   },
   book={
      series={Lecture Notes in Math.},
      volume={992},
      publisher={Springer},
      place={Berlin},
   },
   date={1983},
   pages={101--114},
}
		



\bib{KozonoWadade2008}{article}{
   author={Kozono, Hideo},
   author={Wadade, Hidemitsu},
   title={Remarks on Gagliardo--Nirenberg type inequality with critical
   Sobolev space and BMO},
   journal={Math. Z.},
   volume={259},
   date={2008},
   number={4},
   pages={935--950},
   issn={0025-5874},
   doi={10.1007/s00209-007-0258-5},
}

\bib{Ledoux2003}{article}{
   author={Ledoux, M.},
   title={On improved Sobolev embedding theorems},
   journal={Math. Res. Lett.},
   volume={10},
   date={2003},
   number={5-6},
   pages={659--669},
   issn={1073-2780},
}

\bib{Lions1984CC2}{article}{
   author={Lions, P.-L.},
   title={The concentration-compactness principle in the calculus of
   variations. The locally compact case. II},
   journal={Ann. Inst. H. Poincar\'e Anal. Non Lin\'eaire},
   volume={1},
   date={1984},
   number={4},
   pages={223--283},
   issn={0294-1449},
}



\bib{Lions1985CC1}{article}{
   author={Lions, P.-L.},
   title={The concentration-compactness principle in the calculus of
   variations. The limit case},
   part={I},
   journal={Rev. Mat. Iberoamericana},
   volume={1},
   date={1985},
   number={1},
   pages={145--201},
   issn={0213-2230},
   doi={10.4171/RMI/6},
}

\bib{MachiharaOzawa2003}{article}{
   author={Machihara, Shuji},
   author={Ozawa, Tohru},
   title={Interpolation inequalities in Besov spaces},
   journal={Proc. Amer. Math. Soc.},
   volume={131},
   date={2003},
   number={5},
   pages={1553--1556},
   issn={0002-9939},
}

\bib{Mazya2010}{book}{
   author={Maz'ya, Vladimir},
   title={Sobolev spaces with applications to elliptic partial differential
   equations},
   series={Grundlehren der Mathematischen Wissenschaften 
},
   volume={342},
   edition={2},
   publisher={Springer},
   place={Heidelberg},
   date={2011},
   pages={xxviii+866},
   isbn={978-3-642-15563-5},
}

\bib{MazyaShaposhnikova1999}{article}{
   author={Maz{\cprime}ya, Vladimir},
   author={Shaposhnikova, Tatyana},
   title={On pointwise interpolation inequalities for derivatives},
   journal={Math. Bohem.},
   volume={124},
   date={1999},
   number={2-3},
   pages={131--148},
   issn={0862-7959},
}

\bib{MazyaShaposhnikova2002}{article}{
   author={Maz{\cprime}ya, V.},
   author={Shaposhnikova, T.},
   title={On the Bourgain, Brezis, and Mironescu theorem concerning limiting
   embeddings of fractional Sobolev spaces},
   journal={J. Funct. Anal.},
   volume={195},
   date={2002},
   number={2},
   pages={230--238},
   issn={0022-1236},
}

\bib{MeyerRiviere2003}{article}{
   author={Meyer, Yves},
   author={Rivi{\`e}re, Tristan},
   title={A partial regularity result for a class of stationary Yang-Mills
   fields in high dimension},
   journal={Rev. Mat. Iberoamericana},
   volume={19},
   date={2003},
   number={1},
   pages={195--219},
   issn={0213-2230},
}

\bib{MorozVanSchaftingen2013Ground}{article}{
   author={Moroz, Vitaly},
   author={Van Schaftingen, Jean},
   title={Groundstates of nonlinear Choquard equations: Existence,
   qualitative properties and decay asymptotics},
   journal={J. Funct. Anal.},
   volume={265},
   date={2013},
   number={2},
   pages={153--184},
   issn={0022-1236},
}

\bib{Nirenberg1959}{article}{
      author={Nirenberg, L.},
       title={On elliptic partial differential equations},
        date={1959},
     journal={Ann. Scuola Norm. Sup. Pisa (3)},
      volume={13},
       pages={115\ndash 162},
}

\bib{ONeil1963}{article}{
   author={O'Neil, Richard},
   title={Convolution operators and $L(p,\,q)$ spaces},
   journal={Duke Math. J.},
   volume={30},
   date={1963},
   pages={129--142},
   issn={0012-7094},
}
		



\bib{PatalucciPisante}{article}{
  author = {Patalucci, Giampierro},
  author = {Pisante, Adriano},
  title = {Improved Sobolev embeddings, profile decomposition, and concentration-compactness
for fractional sobolev spaces},
  eprint={arXiv:1302.5923},
}
  
\bib{Peetre1966}{article}{
   author={Peetre, Jaak},
   title={Espaces d'interpolation et th\'eor\`eme de Soboleff},
   journal={Ann. Inst. Fourier (Grenoble)},
   volume={16},
   date={1966},
   number={1},
   pages={279--317},
   issn={0373-0956},
}

\bib{Peetre1969}{article}{
   author={Peetre, Jaak},
   title={On the theory of \(\mathcal{L}_{p, \lambda }\) spaces},
   journal={J. Functional Analysis},
   volume={4},
   date={1969},
   pages={71--87},
}

\bib{RafeiroSamkoSamko2013}{article}{
   author={Rafeiro, Humberto}, 
   author={Samko, Natasha},
   author={Samko, Stefan},
   title={Morrey-Campanato Spaces: an Overview},
   book={
      editor={Karlovich, Yuri I.}, 
      editor={Rodino, Luigi},
      editor={Silbermann, Bernd},
      editor={Spitkovsky,Ilya M.},
      title={Operator Theory, Pseudo-Differential Equations, and Mathematical Physics},
      subtitle={The Vladimir Rabinovich Anniversary Volume},
      publisher={Springer},
      series={Operator Theory: Advances and Applications},
      volume={228}, 
      place={Basel},
   },
   date={2013}, 
   pages={293--323},
}

\bib{Stein1970}{book}{
   author={Stein, Elias M.},
   title={Singular integrals and differentiability properties of functions},
   series={Princeton Mathematical Series},
   volume = {no. 30},
   publisher={Princeton University Press},
   place={Princeton, N.J.},
   date={1970},
   pages={xiv+290},
}

\bib{Strzelecki2006}{article}{
   author={Strzelecki, P.},
   title={Gagliardo--Nirenberg inequalities with a BMO term},
   journal={Bull. London Math. Soc.},
   volume={38},
   date={2006},
   number={2},
   pages={294--300},
   issn={0024-6093},
}


		
\bib{Tartar1998}{article}{
   author={Tartar, Luc},
   title={Imbedding theorems of Sobolev spaces into Lorentz spaces},
   journal={Boll. Unione Mat. Ital. Sez. B Artic. Ric. Mat. (8)},
   volume={1},
   date={1998},
   number={3},
   pages={479--500},
   issn={0392-4041},
}

\bib{TransiricoTroisiVitolo1995}{article}{
   author={Transirico, M.},
   author={Troisi, M.},
   author={Vitolo, A.},
   title={Spaces of Morrey type and elliptic equations in divergence form on
   unbounded domains},
   journal={Boll. Un. Mat. Ital. B (7)},
   volume={9},
   date={1995},
   number={1},
   pages={153--174},
}

\bib{Triebel1978}{book}{
   author={Triebel, Hans},
   title={Interpolation theory, function spaces, differential operators},
   series={North-Holland Mathematical Library},
   volume={18},
   publisher={North-Holland},
   place={Amsterdam},
   date={1978},
   pages={528},
   isbn={0-7204-0710-9},
}

\bib{Willem1996}{book}{
    author = {Willem, Michel},
     title = {Minimax theorems},
    series = {Progress in Nonlinear Differential Equations and their Applications, 24},
 publisher = {Birkh\"auser Boston Inc.},
   address = {Boston, Mass.},
      year = {1996},
     pages = {x+162},
      isbn = {0-8176-3913-6}
}


\bib{YuanSickelYang2010}{book}{
   author={Yuan, Wen},
   author={Sickel, Winfried},
   author={Yang, Dachun},
   title={Morrey and Campanato meet Besov, Lizorkin and Triebel},
   series={Lecture Notes in Mathematics},
   volume={2005},
   publisher={Springer},
   place={Berlin},
   date={2010},
   pages={xii+281},
   isbn={978-3-642-14605-3},
}


\end{biblist}
\end{bibdiv}
\end{document}